\documentclass{article}
\usepackage{amsmath,amsthm,amssymb,amscd}

\newcommand{\ga}{\alpha}
\newcommand{\gb}{\beta}

\newcommand{\gd}{\delta}
\newcommand{\gw}{\omega}

\newcommand{\gs}{\sigma}
\newcommand{\eps}{\varepsilon}

\newcommand{\coll}{\mathrm{Coll}}

\newcommand{\liff}{\leftrightarrow}

\newcommand{\cantor}{2^\gw}

\newcommand{\dothgen}{\dot h_{\mathit{gen}}}

\newcommand{\supp}{\mathrm{supp}}

\newcommand{\dom}{\mathrm{dom}}

\newcommand{\rng}{\mathrm{rng}}

\newtheorem{theorem}{Theorem}[section]

\newtheorem{corollary}[theorem]{Corollary}

\newtheorem{proposition}[theorem]{Proposition}

\theoremstyle{definition}
\newtheorem{definition}[theorem]{Definition}
\newtheorem{example}[theorem]{Example}

\newtheorem{conjecture}[theorem]{Conjecture}

\title{Coloring the distance graphs in three dimensions\footnote{2010 AMS subject classification 03E15, 03E25, 03E35.}}

\author{
Jind{\v r}ich Zapletal\\
University of Florida\\
zapletal@ufl.edu}

\begin{document}
\maketitle

\begin{abstract}
Let $\Delta_n$ be the graph on $\mathbb{R}^n$ connecting points of rational Euclidean distance. It is consistent with choiceless set theory that $\Delta_3$ has countable chromatic number, yet the chromatic number of $\Delta_4$ is uncountable.
\end{abstract}

\section{Introduction}

Let $n\geq 1$ be a number. Let $\Delta_n$ be the graph on $\mathbb{R}^n$ connecting points of rational Euclidean distance. Komj{\' a}th \cite{komjath:decomposition} proved  that in ZFC, all graphs $\Delta_n$ have countable chromatic number; the cases $n=2$ and $n=3$ are easier and have been known much earlier \cite{komjath:r3} \cite{hajnal:chromatic}. Difficulty of the proofs greatly increases with $n$, which leads to a natural conjecture. Recall that DC, the Axiom of Dependent Choices, is the fragment of the Axiom of Choice sufficient for a smooth development of mathematical analysis which does not imply existence of ``paradoxical'' objects such as countable colorings of graphs in question.

\begin{conjecture}
\label{conjecture1}
Let $n\geq 1$ be a number. The statement ``the chromatic number of $\Delta_n$ is countable while that of $\Delta_{n+1}$ is not"  is consistent with ZF+DC.
\end{conjecture}

\noindent The book \cite[Section 12.3]{z:geometric} affirmed the conjecture in the cases $n=1$ and $n=2$. The conjecture is still open for $n=4$ and higher. In this paper, I prove 

\begin{theorem}
\label{maintheorem}
It is consistent relatively to an inaccessible cardinal that ZFC+DC holds, the chromatic number of $\Delta_3$ is countable, yet that of $\Delta_4$ is not.
\end{theorem}

\noindent The proof uses the approach of geometric set theory \cite{z:geometric}. The model is constructed as a generic extension of the classical choiceless Solovay model \cite[Theorem 26.14]{jech:newset} by a rather canonical coloring poset (Definition~\ref{posetdefinition}) which works for a number of other graphs on Euclidean spaces. The poset can be also used to give a fairly short and conceptual new proof of the conjecture in the case of $n=2$. However, the basic assumptions of the construction of the coloring poset break down in the case of $n\geq 4$, leaving plenty of work for the future. There are many other geometrically natural graphs or hypergraphs which can be defined in Euclidean spaces of various dimensions, and conjectures parallel to~\ref{conjecture1} can be formulated. Many interesting cases will be handled in forthcoming work. The method uses an inaccessible cardinal to support the general framework. Removing the inaccessible cardinal in the spirit of \cite{Sh:176} requires plenty of improvisation, but in the given case is probably possible.
 
The argument follows the standard architecture of proofs in geometric set theory \cite{z:geometric}, which may be long, yet every step brings conceptually significant information to the reader. First, a generalization of mutual genericity is introduced, which is preserved by further mutually generic extension--Section~\ref{generalizationsection}. Then, some forcing extensions which satisfy this generalization are exhibited--Section~\ref{exampleIsection}. The conjunction of those two items leads to a strengthening of the notion of balance in Suslin posets and preservation theorems for extensions of the choiceless Solovay model by Suslin posets satisfying it--Section~\ref{preservationsection}. Finally, one has to construct interesting Suslin posets which possess balanced virtual conditions of this type. This is the key part of the argument, always leading to some natural amalgamation questions reminiscent of model theory--Section~\ref{exampleIIsection}. The confluence of these four rivers provides enough flotation to lift the ship of the final brief argument at the beginning of Section~\ref{exampleIIsection}.

The notation of the paper uses the set theoretic standard of \cite{jech:newset}, and in matters of geometric set theory, \cite{z:geometric}. All theorems and definitions take place in ZFC set theory.

\section{A generalization of mutual genericity}
\label{generalizationsection}

In this section, I introduce a combinatorially motivated variation of mutual genericity. Recall that if $m, n\geq 1$ are numbers then $K_{m, n}$ is the complete bipartite graph in which the sizes of the bipartition are $m$ and $n$ respectively.

\begin{definition}
\label{knmdefinition}
Let $n, m\geq 1$ be numbers. Let $V[G_0], V[G_1]$ be generic extensions inside some ambient forcing extension of the ground model $V$. Say that $V[G_0], V[G_1]$ are \emph{$K_{m, n}$-transcendental} if $V[G_0]\cap V[G_1]=V$ and whenever $\Gamma$ is a closed graph on a Polish space $X$ in $V$ and $x_0\in X\cap V[G_0]\setminus V$ and $x_1\in X\cap V[G_1]\setminus V$ are $\Gamma$-connected points, then $\Gamma$ contains an injective homomorphic image of $K_{m,n}$.
\end{definition}

\noindent I will need to show that this notion of transcendence behaves well with respect to mutuall generic extensions. This is encapsulated in the following easy but critical proposition.

\begin{proposition}
\label{productproposition}
Let $m, n\geq 1$ be numbers. Let $V[G_0], V[G_1]$ be $K_{m,n}$ -transcendental pair of generic extensions. Let $P_0\in V[G_0]$ and $P_1\in V[G_1]$ be posets and $H_0\subset P_0$ and $H_1\subset P_1$ be filters mutually generic over the model $V[G_0, G_1]$. Then $V[G_0][H_0], V[G_1][H_1]$ is a $K_{m, n}$-transcendental pair of extensions.
\end{proposition}

\begin{proof}
The fact that $V[G_0][H_0]\cap V[G_1][H_1]=V$ uses mutual genericity in a straighforward way, and it is left to the reader. The verification of $K_{m, n}$-transcendence is slightly more difficult. Let $\Gamma$ be a closed graph on a Polish space $X$ in the ground model which does not contain an injective homomorphic image of $K_{m, n}$. Move to the model $V[G_0, G_1]$.
Suppose towards a contradiction that $p_0\in P_0$ and $p_1\in P_1$ are conditions and $\tau_0$, $\tau_1$ are $P_0$- and $P_1$-names in $V[G_0]$ and $V[G_1]$ respectively for elements of $X$ which do not belong to $V$, and $\langle p_0, p_1\rangle$ forces in $P_0\times P_1$ that $\tau_0\mathrel\Gamma\tau_1$ holds. The contradiction is derived in two separate cases.

\noindent\textbf{Case 1.} $p_0\Vdash\tau_0\notin V[G_0]$ and $p_1\Vdash\tau_1\notin V[G_1]$. Strengthening the conditions if necessary, it is possible to find disjoint basic open sets $O_0, O_1\subset X$ such that $p_0\in O_0$ and $p_1\in O_1$. Working in the respective models $V[G_0]$ and $V[G_1]$ find countable elementary submodels $M_0$ and $M_1$ of large structures, containing $p_0, \tau_0$ and $p_1, \tau_1$ respectively, and filters $h_{0i}\subset P_0$ for $i\in m$ and $h_{0j}\subset P_1$ for $j\in n$ which are mutually generic over the models $M_0$ and $M_1$ respectively and contain the conditions $p_0$ and $p_1$. By the case assumption and mutual genericity of these filters, the points $x_{0i}=\tau/h_{0i}$ and $x_{1j}=\tau_1/h_{1j}$ re pairwise distinct. For every pair of indices $i\in m$ and $j\in m$ it is the case that $x_{0i}\mathrel\Gamma x_{1j}$ holds. To see this, note that the graph $\Gamma$ is closed, and if this failed there would have to be basic open sets $O'_0, O'_1\subset X$ and conditions $p'_0\leq p_0$ and $p'_1\leq p_1$ in the filters $h_{0i}$ and $h_{1i}$ such that $O'_0\times O'_1\times\Gamma=0$ and $p'_{0i}\Vdash \tau_0\in O'_0$ and $p'_{1j}\Vdash\tau_{1j}\in O'_1$. This would contradict the initial assumptions on the names $\tau_0, \tau_1$.

Now, it is clear that the points $x_{0i}$ for $i\in m$ and $x_{1j}$ for $j\in n$ form an injective homomorphic image of $K_{m, n}$ in $\Gamma$, contradicting the initial assumptions on $\Gamma$.

\noindent\textbf{Case 2.} Case 1 fails. Strengthening the conditions if necessary and by symmetry we may assume that there is a point $x_0\in X\cap V[G_0]$ such that $p_0\Vdash\tau_0=\check x_0$ and that there is a basic open set $O_1\subset X$ which does not contain $x_0$ such that $p_1\Vdash\tau_1\in O_1$. By the initial assumptions, $x_0\in V[G_0]\setminus V$ must hold. It follows that $p_1\Vdash\tau_1\notin V[G_1]$ because no point in $V[G_1]\setminus V$ is $\Gamma$-connected with $x_0$ and $\tau_1$ is forced not to belong to $V$. Work in the model $V[G_1]$.

Let $M$ be a countable elementary submodel of some large structure containing $p_1, \tau_1$ and let $h_{1j}\subset P_1$ for $j\in n$ be mutually generic filters over the model $M$ containing the condition $p_1$. Let $x_{1j}=\tau_1/h_{1j}\in X$. As in the previous case, the points $x_{1j}$ are all pairwise distinct and $\Gamma$-connected to $x_0$.

By the transcendence assumption on the models $V[G_0]$ and $V[G_1]$, it must be the case that the points $x_{0j}$ for $j\in n$ all belong to $V$. Now, working in $V$, the closed set $C=\{x\in X\colon \forall j\in n\ x\mathrel\Gamma x_{0j}\}\subset X$ must be uncountable, since in some generic extension (namely $V[G_0]$) it contains a point (namely $x_0$) which is not in $V$. This means that $\Gamma$ contains an injective homomorphic copy of $K_{m, n}$, contradicting the initial assumptions of $\Gamma$ again.
\end{proof}

\begin{corollary}
Let $m, n\geq 1$ be numbers. Mutually generic extensions are $K_{m, n}$-transcendental.
\end{corollary}

\begin{proof}
Just use the proposition with the initial pair $\langle V, V\rangle$ of models which is $K_{m,n}$-transcendental for trivial reasons.
\end{proof}

\noindent The conclusion of the corollary is in fact very weak. The following easy proposition gives a much stronger conclusion which is needed later in the paper.

\begin{proposition}
\label{perfproposition}
Let $\Gamma$ be a closed graph on a Polish space $X$. The following are equivalent:

\begin{enumerate}
\item there are perfect sets $C_0, C_1\subset X$ such that $C_0\times C_1\subset\Gamma$;
\item there are mutually generic extensions $V[G_0]$, $V[G_1]$ and $\Gamma$-connected points $x_0\in X\cap V[G_0]\setminus V$ and $x_1\in X\cap V[G_1]\setminus V$.
\end{enumerate}
\end{proposition}

\begin{proof}
For (1)$\to$(2), just let $P_0, P_1$ be Cohen forcings for $C_0, C_1$--the posets of nonempty relatively open subsets of $C_0$ and $C_1$ ordered by inclusion. They add respective generic points $\tau_0\in C_0$ and $\tau_1\in C_1$ which do not belong to $V$. Moreover, $P_0\times P_1\vdash \tau_0\mathrel\Gamma\tau_1$ since $C_0\times C_1\subset\gamma$ holds in the $P_0\times P_1$ extension by a Mostowski absoluteness argument.

(2)$\to$(1) is slightly more difficult. Suppose that $P_0, P_1$ are posets, $\tau_0$, $\tau_1$ are respective $P_0$- and $P_1$-names for elements of $X\setminus V$ such that $P_0\times P_1\Vdash\tau_0\mathrel\Gamma\tau_1$. Let $M$ be a countable elementary submodel of a large substructure containing in particular $P_0, P_1, \Gamma, \tau_0$ and $\tau_1$.
Consider the sets $A_0=\{x\in X\colon$ there is a filter $g\subset P_0$ generic over $M$ such that $\tau_0/g=x\}$ and $A_1=\{x\in X\colon$ there is a filter $g\subset P_1$ generic over $M$ such that $\tau_1/g=x\}$. These are analytic subsets of $X$ by their definitions. Since $\tau_0$ and $\tau_1$ are forced not to belong to the ground model, a simple diagonalization argument shows that both sets $A_0, A_1$ are uncountable.

 It is also the case that $A_0\times A_1\subset\Gamma\cup$the diagonal. If this failed, there would have to be filters $g_0\subset P_0$ and $g_1\subset P_1$ generic over $M$ such that the points $x_0=\tau_0/g_0$ and $x_1=\tau_1/g_1$ are distinct and $\Gamma$-disconnected. Since the graph $\Gamma$ is closed, there must be basic open sets $O_0, O_1\subset X$ such that $x_0\in O_0$ and $x_1\in O_1$ and $O_0\times O_1\cap\Gamma=0$. By the forcing theorem applied in $M$, there have to be conditions $p_0\in g_0$ and $p_1\in g_1$ such that $p_0\Vdash\tau_0\in O_0$ and $p_1\Vdash\tau_1\in O_1$ holds in $M$. Since $M$ is an elementary submodel, the two forcing statements hold in $V$. In conclusion, the pair $\langle p_0, p_1\rangle$ forces in the product that $\tau_0, \tau_1$ are not $\Gamma$-connected. This contradicts the initial assumptions on $\tau_0, \tau_1$.

Now, to complete the proof of (2), just use the perfect set theorem for analytic sets to find disjoint perfect sets $C_0\subset A_0$ and $C_1\subset A_1$.
\end{proof}

\section{Examples I}
\label{exampleIsection}

In this section, I provide two examples of $K_{m, n}$-transcendental extensions. They depend on simple preliminary definitions and considerations which are not stated in their strongest possible form. Throughout this section, if $X$ is a Polish space the $P_X$ is the poset of nonempty open subsets of $X$ ordered by inclusion. It adds a generic point of $X$ which is the unique point in the intersection of all open sets in the generic filter. There is a useful observation about continuous open maps between Polish spaces: if $f\colon X\to Y$ is a continuous open map, then the $f$-image of the $P_X$-generic point is a $P_Y$-generic point \cite[Proposition 3.1.1]{z:geometric}.

\begin{definition}
\label{pendefinition}
Let $\eps>0$ be a real number and $n\geq 1$ be a number. Let $X_{\eps n}\subset [\mathbb{R}^n]^2$ be the closed set consisting of pairs of points of Euclidean distance $\eps$. Let $P_{\eps n}$ be its associated Cohen poset of nonempty open subsets of $X_{\eps n}$ ordered by inclusion. The poset adds a generic pair whose coordinates are denoted by $\dot x_0$ and $\dot x_1$.
\end{definition}

\begin{proposition}
\label{penproposition}
Let $n\geq 1$ and let $\eps>0$ be a real number. The poset $P_{\eps n}$ forces $\dot x_0, \dot x_1$ are separately Cohen generic elements of the space $\mathbb{R}^n$.
\end{proposition}

\begin{proof}
This follows from the fact that both projection maps from $X_{\eps n}$ to $\mathbb{R}^n$ are open by \cite[Proposition 3.1.1]{z:geometric}. The openness of these maps is immediate and left to the reader.
\end{proof}

\begin{definition}
\label{isocopydefinition}
Let $\langle M, d\rangle$ be a finite metric space and $n\geq 1$ be a number such that $M$ can be isometrically embedded in $\mathbb{R}^n$. Let $X_{Mn}$ be the space of all isometric embeddings from $M$ to $\mathbb{R}^n$ equipped with the topology inherited from $(\mathbb{R}^n)^M$. The Cohen poset $P_{Mn}$ associated with it introduces a generic isometric embedding $\dothgen\colon M\to \mathbb{R}^n$.
\end{definition}

\begin{proposition}
\label{genproposition}
Let $m_0, m_1\in M$ be distinct points of $d$-distance $\eps$. Then $P_{Mn}$ forces the pair $\dothgen(m_0), \dothgen(m_1)$ to be $P_{\eps n}$-generic.
\end{proposition}

\noindent In fact, a more careful argument with a small infusion of rigidity theory will show that if $N\subset M$ is any subset then $\dothgen\restriction N$ is forced to be a $P_{Nn}$-generic embedding of $N$ into $\mathbb{R}^n$.

\begin{proof}
This again follows from the fact that the projection map from $X_{Mn}$ to $X_{\eps n}$ given by the coordinates $m_0, m_1$ is open, together with \cite[Proposition 3.1.1]{z:geometric}. The openness of this map is immediate and left to the reader.
\end{proof}

\noindent Finally, I am ready to investigate transcendence properties of extensions generated by the posets $P_{\eps n}$ for various dimensions $n\geq 2$.

\begin{example}
\label{2dexample}
Let $\eps>0$ be a real number and $n\geq 2$. Let $x_0, x_1\in\mathbb{R}^n$ be $P_{\eps n}$-generic pair of elements of $\mathbb{R}^n$. Then $V[x_0]\cap V[x_1]=V$.
\end{example}

\begin{proof}
If this failed, there would have to be an ordinal $\ga$, $P_{\mathbb{R}^n}$-names $\tau_0, \tau_1$ for subsets of $\ga$ which do not belong to $V$, and a condition $p\in P_{\eps n}$ such that  $p\Vdash\tau_0/\dot x_0=\tau_1/\dot x_1$. Strengthening $p$ further, we may find disjoint basic open sets $O_0, O_1\subset \mathbb{R}^n$ such that $p=(O_0\times O_1)\cap X_{\eps n}$.

Find disjoint open sets $O_{00}, O_{01}\subset O_0$ such that for all points $y_{00}\in O_{00}$ and $y_{01}\in O_{01}$ there is a point $y_1\in O_1$ of Euclidean distance $\eps$ from both. Find an ordinal $\gb\in\ga$ and nonempty open sets of $O_{00}$ which decide the statement $\check\gb\in\tau_0$ in a different way, and an open subset of $O_{01}$ deciding the statement $\check\gb\in\tau_0$ in the poset $P_{\mathbb{R}^n}$. It is now clear that one can thin down the open sets $O_{00}$ and $O_{01}$ if necessary so that they decide the statement $\check\gb\in\tau_0$ in opposite ways. 

Find points $y_{00}\in O_{00}$ and $y_{01}\in O_{01}$ and $y_1\in O_1$ such that the last mentioned point has distance $\eps$ from the previous two. Use the poset of Definition~\ref{isocopydefinition} to force a generic isometric copy of the triple $\langle y_{00}, y_{01}, y_1\rangle$ in the open set $O_{00}\times O_{01}\times O_1$. Let $z_{00}, z_{01}, z_1\rangle$ be the generic isometric copy. By Proposition~\ref{genproposition}, The pairs $\langle z_{00}, z_1\rangle$ and $\langle z_{01}, z_1\rangle$ are both $P_{\eps n}$-generic and meet the condition $p$. By the forcing theorem, it should be the case that $\tau_0/z_{00}=\tau_1/z_1=\tau_{0}/z_{01}$. This is impossible as the first and last set differ in the membership of the ordinal $\gb$ in them.
\end{proof}

\begin{example}
\label{3dexample}
Let $\eps>0$ be a real number. Let $x_0, x_1\in\mathbb{R}^3$ be $P_{\eps 3}$-generic pair of elements of $\mathbb{R}^3$. The extensions $V[x_0]$, $V[x_1]$ are $K_{2, n}$-transcendental for every $n\in\gw$.
\end{example}

\begin{proof}
It follows from Example~\ref{2dexample} that $V[x_0]\cap V[x_1]=V$.  Now suppose that $\Gamma$ is a closed graph on a Polish space $Z$ and that some condition $p\in P_{\eps 3}$ forces that there are $\Gamma$-connected points $z_0\in V[\dot x_0]\setminus V$ and $z_1\in V[\dot x_1]\setminus V$. I must find a homomorphic injective copy of $K_{2, n}$ in $\Gamma$. Use Proposition~\ref{penproposition} (1) to strengthen the condition $p$ to find $P_{\mathbb{R}^3}$-names $\tau_0$, $\tau_1$ for elements of $Z\setminus V$ such that $p\Vdash\tau_0/\dot x_0\mathrel\Gamma\tau_1/\dot x_1$. Strengthening $p$ further, we may find disjoint basic open sets $O_0, O_1\subset \mathbb{R}^3$ such that $p=(O_0\times O_1)\cap X_{\eps 3}$.

Now, it is easy to find points $y_{0i}\colon i\in 2$ in $O_0$ and $y_{1j}\colon j\in n$ in $O_1$ such that for each $i\in 2$ and $j\in n$ the distance of $y_{0i}$ from $y_{1j}$ is equal to $\eps$: first find the points $y_{0i}$ close to each other and then place the points $y_{1j}$ in the circle which is the intersection of the $\eps$-spheres around them. Force a generic isometric copy of the set $\{y_{0i}\colon i\in 2, y_{1j}\colon j\in n\}$ in $O_0\cup O_1$ using the poset of Definition~\ref{isocopydefinition}, obtaining points $x_{0i}\in O_0$ for $i\in 2$ and $x_{1j}\in O_1$ for $j\in n$.

By Proposition~\ref{genproposition}, each pair among these points is generic for their respective distance poset. By Example~\ref{2dexample}, the models obtained by adjoining one of these points pairwise intersect in $V$. Thus, the points $\tau_0/x_{0i}$ for $i\in 2$ and $\tau_1/x_{1j}$ for $j\in n$ in the space $Z$ are pairwise distinct, and by the forcing theorem applied to the poset $P_{\eps 3}$, for each $i\in 2$ and each $j\in n$ the points $\tau_0/x_{0i}$ and $\tau_1/x_{1j}$ are $\Gamma$-connected. This shows that the graph $\Gamma$ contains an injective homomorphic copy of $K_{2, n}$.
\end{proof}

\begin{example}
\label{4dexample}
Let $\eps>0$ be a real number. Let $x_0, x_1\in\mathbb{R}^4$ be $P_{\eps 4}$-generic pair of elements of $\mathbb{R}^4$. The extensions $V[x_0]$, $V[x_1]$ are $K_{n, n}$-transcendental for every $n\in\gw$.
\end{example}

\begin{proof}
This follows the lines of the previous proof, with an important additional insight. Consider the sets $C_0, C_1\subset\mathbb{R}^4$ defined by $C_0=\{\langle r_0, r_1, 0, 0\rangle\colon r_0^2+r_1^2=\eps^2/2\}$ and  $C_1=\{\langle 0, 0, r_2, r_3\rangle\colon r_2^2+r_3^2=\eps^2/2\}$. It is clear that any element of $C_0$ has Euclidean distance $\eps$ to each element of $C_1$. Now, suppose $O_0, O_1\subset\mathbb{R}^2$ are open sets such that $X_{\eps 4}\cap O_0\times O_1\neq 0$. Then, it is possible to find an isometry $\pi$ of $\mathbb{R}^4$ such that $\pi''C_0\cap O_0$ and $\pi''C_1\cap O_1$ are both nonempty (and therefore uncountable) sets. Then, choose pairwise distinct points $y_{0i}\in \pi''C_0\cap O_0$ and $y_{1i}\in\pi''C_1\cap O_1$ for $i\in n$ and proceed as in the previous example.
\end{proof}

\section{Preservation theorems}
\label{preservationsection}

Every generalization of mutual genericity such as the ones introduced in Definition~\ref{knmdefinition} comes with an associated notion of balance for Suslin partial orders.

\begin{definition}
\label{balancedefinition}
Let $P$ be a Suslin poset and $m, n\geq 1$ be numbers.

\begin{enumerate}
\item A virtual condition $\bar p$ in $P$ is \emph{$K_{m, n}$-balanced} if for every $K_{m, n}$-transcendental pair $V[G_0]$, $V[G_1]$ of generic extensions and conditions $p_0\in V[G_0]$ and $p_1\in V[G_1]$ in $P$ below $\bar p$, $p_0, p_1$ have a common lower bound.
\item The poset $P$ is \emph{$K_{m, n}$-balanced} if below every condition in $P$ there is a $K_{m, n}$-balanced virtual condition.
\end{enumerate}
\end{definition}

\noindent In this section, I present two preservation theorems for $K_{n, m}$-balanced forcings which are relevant to colorings of rational distance graphs on Euclidean spaces. The parlance follows the standard of geometric set theory as introduced in \cite{z:geometric}, in particular, \cite[Convention 1.7.18]{z:geometric}.

\begin{theorem}
\label{pres1theorem}
Let $n\geq 1$ be any number. In $\gs$-closed, cofinally $K_{n, n}$-balanced extensions of the Solovay model, for every nonmeager set $B\subset \mathbb{R}^4$ there is a positive real number $\eps>0$ such that every positive number smaller than $\eps$ is the distance of some points in $B$.
\end{theorem}

\begin{proof}
Let $\kappa$ be an inaccessible cardinal. Let $P$ be a Suslin poset which is cofinally $K_{n, n}$-balanced below $\kappa$. Let $W$ be the choiceless Solovay model obtained from $\kappa$. Work in $W$. Let $p\in P$ be a condition and $\tau$ be a $P$-name for a non-meager subset of $\mathbb{R}^4$ for which the conclusion fails. Because the poset $P$ is $\gs$-closed, gradually strengthening the condition $p$ it is possible to find a countable set $a\subset\mathbb{R}^+$ converging to zero such that $p$ forces that no two points of $\tau$ have distance in the set $a$. To reach a contradiction, it is necessary to find conditions $p_0, p_1\leq p$ and points $x_0, x_1\in\mathbb{R}^4$ such that the distance between $x_0, x_1$ belongs to $a$, $p_0\Vdash\check x_0\in \tau$, $p_1\Vdash\check x_1\in\tau$, and $p_0, p_1$ have a common lower bound.

To do this, first find a parameter $z\in\cantor$ such that $p, \tau, a$ are definable from $z$. Let $V[K]$ be an intermediate generic extension of $V$ obtained by a poset of cardinality smaller than $\kappa$ such that $z\in V[K]$ and $P$ is $K_{n, n}$-balanced. Work in the model $V[K]$. Let $\bar p\leq p$ be a $K_{n, n}$-balanced condition below $p$. Consider the Cohen forcing $Q$ of nonempty open subsets of $\mathbb{R}^4$, adding a single point $\dot y\in\mathbb{R}^4$ in the intersection of all sets in the generic filter. There must be a condition $q\in Q$, a poset $R$ of cardinality smaller than $\kappa$ and a $Q\times R$-name $\gs$ for a condition in $P$ stronger than $\bar p$ such that $q\Vdash_Q R\Vdash\coll(\gw, <\kappa)\Vdash \gs\Vdash_P\dot y\in\tau$. Otherwise, in the model $W$ the dense $G_\gd$-set of all points $Q$-generic over $V[K]$ would be forced disjoint from $\tau$ by $\bar p$, contradicting the initial assumptions on the name $\tau$.

Work in $W$ again. Let $\eps\in a$ be a number so small that $q$ contains two points of distance $\eps$. Find a pair $\langle y_0, y_1\rangle$ of points in $q$ of distance $\eps$ which is generic over $V[K]$ for the poset $P_{4\eps}$. By Example~\ref{4dexample}, the models $V[K][y_0]$ and $V[K][y_1]$ are $K_{n, n}$-transcendental. Let $H_0, H_1\subset R$ be filters mutually generic over the model $V[K][y_0, y_1]$. By Proposition~\ref{productproposition}, the models $V[K][y_0][H_0]$ and $V[K][y_1][H_1]$ are still $K_{n, n}$-transcendental. Let $p_0=\gs/y_0, H_0$ and $p_1=\gs/y_1, H_1$. These are conditions stronger than $\bar p$ which belong to the respective models and force in $W$ that $\check y_0\in\tau$ and $\check y_1\in\tau$ respectively holds. By the balance of the virtual condition $\bar p$, the conditions $p_0, p_1$ have a lower bound as required.
\end{proof}

\begin{theorem}
\label{pres2theorem}
Let $n\geq 1$ be any number. In $\gs$-closed, cofinally $K_{2, n}$-balanced extensions of the Solovay model, for every nonmeager set $B\subset \mathbb{R}^3$ there is a positive real number $\eps>0$ such that every positive number smaller than $\eps$ is the distance of some points in $B$.
\end{theorem}

\begin{proof}
The argument is nearly literally the same as that for Theorem~\ref{pres2theorem}, with Example~\ref{4dexample} replaced with Example~\ref{3dexample}.
\end{proof}

\section{Examples II}
\label{exampleIIsection}

Here, I provide a coloring poset which works for a great number of  graphs on Euclidean spaces. For the parlance, if $\Gamma$ is a graph on a Polish space $X$, a \emph{perfect biclique} in $\Gamma$ consists of two perfect sets $C_0, C_1\subset X$ such that $C_0\times C_1\subset\Gamma$. A graph $\Gamma$ on a Euclidean space $X$ of dimension $d$ is \emph{algebraic} if there is an algebraic set $\bar\Gamma\subset X^2$ such that for distinct points $x_0, x_1\in X$, $\{x_0, x_1\}\in\Gamma\liff \langle x_0, x_1\rangle\in\bar\Gamma$. Note that there may be points $x\in X$ such that $\langle x, x\rangle\in\bar\Gamma$.

\begin{theorem}
\label{coloringtheorem}
Suppose that $\Gamma$ is a graph on a Euclidean space $X$, such that $\Gamma=\bigcup_i\Gamma_i$ is a union of countably many algebraic graphs neither of which contains a perfect biclique. Then there is a coloring poset $P_\Gamma$ such that

\begin{enumerate}
\item $P_\Gamma$ is Suslin and $\gs$-closed;
\item $P_\Gamma$ is $n, 2$-centered for every natural number $n\in\gw$;
\item the union of the generic filter on $P_\Gamma$ is forced to be a total $\Gamma$-coloring on $X$ with countable range;
\item $P_\Gamma$ is balanced;
\item if $m, n\in\gw$ are numbers such that no graph $\Gamma_i$ contains an injective homomorphic copy of $K_{m, n}$ then $P_\Gamma$ is $K_{m, n}$-balanced.
\end{enumerate}
\end{theorem}

\noindent Here, a poset is $n, 2$-centered if every collection of cardinality $n$ consisting of pairwise compatible conditions has a common lower bound.

The coloring poset can be used to prove Theorem~\ref{maintheorem}. Let $\Delta_3$ be the rational distance graph on $\mathbb{R}^3$. Thus, $\Delta_3=\bigcup_{q\in\mathbb{Q}^+}\Delta_{3q}$ where $\Delta_{3q}$ is the algebraic graph on $\mathbb{R}^3$ connecting points of distance $q$. A brief geometric argument shows that the graphs $\Delta_{3q}$ do not contain an injective homomorphic copy of $K_{3,3}$: three distinct spheres in $\mathbb{R}^3$ with the same radius have intersection of cardinality at most two. Thus, the associated coloring poset $P_\Delta$ is $K_{3,3}$-balanced under the Continuum Hypothesis.  Consider the $P_{\Delta_3}$-extension of the choiceless Solovay model.  As a $\gs$-closed extension of a model of DC, it satisfies DC. It contains a total $\Delta_3$-coloring as the union of the generic filter on $P_{\Delta_3}$. Theorems~\ref{coloringtheorem} and~\ref{pres1theorem} then show that in the $P_{\Delta_3}$-extension, the rational distance graph $\Delta_4$ on $\mathbb{R}^4$ is not countably chromatic: dividing the space $\mathbb{R}^4$ into countably many pieces, one of them has to be nonmeager. That piece contains all distances sufficiently close to zero, so is not a $\Delta_{4}$-anticlique.

A similar argument can be used to show that in the $P_{\Delta_2}$-extension of the Solovay model, $\Delta_2$ has countable chromatic number while $\Delta_3$ does not. Just note that the rational distance graph $\Delta_2$ on $\mathbb{R}^2$ can be expressed as $\bigcup_{q\in\mathbb{Q}^+}\Delta_{2q}$ where $\Delta_{2q}$ is the algebraic graph on $\mathbb{R}^2$ connecting points of Euclidean distance $q$. Note that $\Delta_{2q}$ contains no injective homomorphic copy of $K_{2,3}$: two distinct circles in the plane intersect in at most two points.
Finally, use Theorem~\ref{coloringtheorem} and~\ref{pres2theorem} to show that the graph $\Delta_3$ is not countably chromatic in the $P_{\Delta_2}$-extension of the choiceless Solovay model.

\subsection{The construction}

The definition of the coloring poset requires a choice of a Borel ideal on $\gw$ which contains all finite sets and which is not generated by countably many sets. Beyond this requirement the choice of the ideal is immaterial. One example is the ideal of sets of asymptotic density zero. Let $\mathcal{I}$ be such an ideal on $\gw$. If $F\subset\mathbb{R}$ is a set and $A\subset X$ is an algebraic set, say that the set $A$ is \emph{visible from} $F$ if there is a polynomial defining the set $A$ whose coefficients belong to $F$.

\begin{definition}
\label{posetdefinition}
Let $\Gamma=\bigcup_i\Gamma_i$ be a countable union of algebraic graphs on a Euclidean space $X$ such that no graph $\Gamma_i$ contains a perfect biclique. The poset $P_\Gamma$ consists of all partial functions $p\colon X\to\gw\times\gw$ such that 

\begin{enumerate}
\item there is a countable real closed subfield $\supp(p)$ such that $\dom(p)=\supp(p)^d$, where $d$ is the dimension of $X$;
\item $p$ is a $\Gamma$-coloring;
\item for every uncountable algebraic set $A\subset X$ visible from $\supp(p)$, the set $\{i\in\gw\colon \{x\in A\colon p(x)=i\}$ is finite$\}$ belongs to $\mathcal{I}$. 
\end{enumerate}

\noindent The ordering is that of reverse inclusion.
\end{definition}

\noindent Note that the definition of the poset does not depend on the algebraic decomposition of the graph $\Gamma$. It  does depend on the fact that one is working in the context of a Polish space $X$ on which the notion of algebraicity makes sense. It is immediately clear that the ordering $\leq$ is transitive and $\gs$-closed. For further analysis of the poset I need several basic facts from algebraic geometry encapsulated in the following paragraph.

 Every infinite algebraic subset of $X$ is in fact uncountable. It follows that if $F$ is a real closed subfield of $\mathbb{R}$ and $A\subset X$ is an algebraic set visible from $F$ which contains some point which is not in $F^d$, then $A$ is uncountable. To see this, recall that the theory of real closed fields is model complete \cite[Corollary 3.3.16]{marker:book}, and therefore $F$ is an elementary submodel of $\mathbb{R}$. By an elementarity argument, $F^d$ must contain points of $A$ arbitrarily close to $x$, which means that $A$ is infinite and therefore uncountable. Finally, there is no infinite sequence of algebraic sets strictly descending with respect to inclusion. This follows from the Hilbert basis theorem \cite[Theorem 3.2.5]{marker:book}. As an important application, if $F\subset\mathbb{R}$ and $A\subset X$ are sets, then there is an inclusion-smallest algebraic set $B\subset X$ visible from $F$ which contains $A$ as a subset. This occurs since the collection of algebraic sets visible from $F$ is closed under intersection, and the search for smaller and smaller algebraic sets visible from $F$ and containing $A$ as a subset must stabilize after finitely many steps. 

I need the following list-coloring result about $\Gamma$.

\begin{proposition}
\label{listproposition}
Let $G\colon X\to\mathcal{I}$ be a function, and let $C\subset X$ be a countable set. Then there is a countable real closed field $F$ and a $\Gamma$-coloring $p\colon F^d\to\gw$ such that

\begin{enumerate}
\item $C\subset F^d$;
\item $\forall x\in F^d\ p(x)\notin G(x)$;
\item for every uncountable algebraic set $A\subset X$ visible from $\supp(p)$, the set $\{i\in\gw\colon \{x\in A\setminus C\colon p(x)=i\}$ is finite$\}$ belongs to $\mathcal{I}$.
\end{enumerate}
\end{proposition}

\begin{proof}
Let $M$ be a countable elementary submodel of a large structure and let $F=M\cap\mathbb{R}$.
The construction of $p$ uses several bookkeeping devices.
Let $b\subset\gw$ be a set in the ideal $\mathcal{I}$ which is not modulo finite covered by any set $G(x)$ for $x\in X\cap M$. Let $\langle x_j\colon j\in\gw\rangle$ be an enumeration of $X\cap M$. Let $\langle A_j\colon j\in\gw\rangle$ be the enumeration of all algebraic sets visible from $M\cap\mathbb{R}$ which are uncountable and minimal such with respect to inclusion. For each $j\in\gw$ consider the set $b_j=\{n\in\gw\colon$ for all but countably many $x\in A_j$, $n\in G(x)\}$ and observe that the set $b_j$ must be in the ideal $\mathcal{I}$ since for all but countably many $x\in A_j$, $b_j\subset G(x)$ holds. Let also $\langle \langle k_j, v_j\rangle\colon j\in\gw\rangle$ be an enumeration of $\gw\times\gw$ with infinite repetitions.

By recursion on $j\in\gw$ build finite maps $q_j\colon X\cap M\to\gw$ such that 

\begin{itemize}
\item $0=q_0\subset q_1\subset\dots$ are $\Gamma$-colorings;
\item $x_j\in\dom(q_{j+1})$;
\item if $k_j<j$ and $v_j\notin b_{k_j}\cup b\cup\rng(q_{k_j})$ then there is a point $z\in (A_{k_j}\cap\dom(q_{j+1}))\setminus (C\cup\dom(q_j))$ such that $q_{j+1}(z)=v_j$;
\item whenever $z\in\dom(q_j)$ is a point and $k\in j$ is a number then either $q_j(z)\in b_k\cup b\cup\rng(q_k)$ or for all $i\in\gw$ there is $y\in A_k$ which is not $\Gamma_i$-connected to $z$.
\end{itemize}

\noindent The first three items will be instrumental in the end, the last item just keeps the recursion going. To make the recursion step, suppose that $q_j$ has been defined and the recursion hypotheses hold.
The map $q_{j+1}$ is constructed in two steps. In the first step, one point is added to $\dom(q_j)$ so that the third item is satisfied. Write $k=k_j$ and $v=v_j$ and assume that $k<j$ and $v\notin b_{k}\cup b\cup\rng(q_{k})$. Consider the following subsets of $A_k$: $B_{li}=\{z\in A_k\colon \forall y\in A_l\ z\mathrel\Gamma_i y\}$ for all $l, i\in\gw$ and sets $C_{yi}=\{z\in A_k\colon z\mathrel \Gamma_i y\}$ for all $y\in\dom(q_j)$ such that $q_j(y)=v$ and all $i\in\gw$. The key point is that the sets $B_{li}$ and $C_{yi}$ are all countable.
This occurs because all of them are algebraic proper subsets of $A_k$, and $A_k$ is a minimal uncountable algebraic set. To see that $B_{li}\neq A_k$, note that otherwise $A_l\times A_k$ would be a perfect biclique in the graph $\Gamma_i$, which is excluded by the assumptions of the theorem. To see that $C_{yi}\neq A_k$, consult the last item of the recursion hypothesis with the observation that $v\notin b_k$.

By the uncountability of the set $A_k$, the definition of the set $b_k$, and elementarity of the model $M$, there must be a point $z\in (A_k\cap M)\setminus (\bigcup_{li} B_{li}\cup\bigcup_{yi}C_{yi}\cup\dom(q_j)\cup C)$ such that $v\notin G(z)$. Let $p_{j+1}(z)=v$ and observe that all recursion hypotheses are still satisfied. There are no new monochromatic edges since $z\notin C_{yi}$ holds, and the last item of the recursion hypothesis is satisfied as $z\notin B_{li}$ holds.

In the second step of the construction of $q_{j+1}$, $x_j$ is added to $\dom(p_j)$. If $x_j\in\dom(q_j)$ then there is nothing to do. Otherwise, let $q_{j+1}(x_j)$ be some element of $b\setminus G(x_j)$ which does not appear in $\rng(q_j)$.
All the recursion hypotheses are still satisfied for trivial reasons.

Once the recursion has been performed, let $p=\bigcup_jq_j$. The first item of the recursion hypothesis shows that $p$ is a $\Gamma$-coloring, and the third item shows that its domain is $F^d$. Now, let $A\subset X$ be an uncountable algebraic set visible from $F$. Since algebraic sets form a Noetherian topology on $X$, there must be an algebraic subset of $A$ which is uncountable and inclusion-minimal such. By elementarity of $M$, there must be such a minimal uncountable algebraic set visible in the field $F$. Therefore, there is $k\in\gw$ such that $A_k\subset A$. Then the set $\{i\in\gw\colon \{x\in A_k\setminus C\colon p(x)=i\}$ is finite$\}$ is a subset of $b\cup b_k\cup\rng(q_k)$ by the third item of the recursion hypothesis. Thus, the coloring $p$ is as required.
\end{proof}

\noindent It is now possible to provide a precise and generous characterization of compatibility in $P_\Gamma$.

\begin{corollary}
\label{keycorollary}
Let $\Gamma=\bigcup_i\Gamma_i$ be a countable union of algebraic graphs on a Euclidean space $X$ such that no graph $\Gamma_i$ contains a perfect biclique. 
Let $a\subset P_\Gamma$ be a finite set of conditions. The following are equivalent:

\begin{enumerate}
\item $a$ has a common lower bound in $P_\Gamma$;
\item for every $z\in X$, $a$ has a common lower bound $q$ such that $z\in\dom(q)$;
\item $\bigcup a$ is a function and a $\Gamma$-coloring.
\end{enumerate}
\end{corollary}

\begin{proof}
It is clear that (2) implies (1) which in turn implies (3). (3) implies (2) is the significant implication. Assume (3) holds and define a function $G\colon X\to\mathcal{I}$ by letting $G(x)=0$ if $x\in\dom(\bigcup a)$, and $G(x)=\gw\setminus \bigcap \{p''A\colon p\in a$ and $A$ is the smallest algebraic set visible in $\supp(p)$ containing $x\}$. Use Proposition~\ref{listproposition} with $C=\dom(\bigcup a)\cup \{z\}$ to find a coloring $q'$. Let $q$ be the coloring with the same domain as $q'$ and same values except for all $y\in\dom(\bigcup a)$ let $q(y)=(\bigcup a)(y)$. This is the desired common lower bound of the set $a$. I will just check that $q$ is indeed a $\Gamma$-coloring. If not, then there must be a condition $p\in a$, a point $y\in\dom(p)$, a point $x\in\dom(q\setminus\bigcup a)$ and an index $i\in\gw$ such that $y\mathrel\Gamma_i x$ and $p(y)=q(x)$. Then the $\Gamma_i$-neighborhood of $y$ is an algebraic set visible from $\supp(p)$ containing the point $x$. By the definition of the function $G$ and the fact that $q(x)\notin G(x)$ it follows that there is a point $x'\in\dom(p)$ which is $\Gamma_i$-connected with $y$ with the same color as $y$. This contradicts the assumption that $p$ is a $\Gamma_i$-coloring.
\end{proof}

\begin{corollary}
Let $n\geq 2$ be a natural number. The poset $P_\Gamma$ is $n, 2$-centered.
\end{corollary}

\begin{proof}
If the item (3) of Proposition~\ref{listproposition} fails for some set $a\subset P_\Gamma$, then it fails for a subset of $a$ of cardinality two.
\end{proof}

\begin{corollary}
$P_\Gamma$ is a Suslin forcing.
\end{corollary}

\begin{proof}
It is clear that $P_\Gamma$ is (or can be presented as) a Borel set and $\leq$ is a Borel relation on it. To check the Borelness of the compatibility relation on $P_\Gamma$, just refer to Corollary~\ref{keycorollary}.
\end{proof}

\begin{corollary}
\label{coloringcorollary}
The poset $P_\Gamma$ forces the union of the generic filter to be a total $\Gamma$-coloring.
\end{corollary}

\begin{proof}
It is only necessary to verify that for every condition $p\in P_\Gamma$ and every point $z\in X$ there is a condition $q\leq p$ such that $z\in\dom(q)$. This is clear from Proposition~\ref{listproposition} applied to the set $a=\{p\}$.
\end{proof}

\subsection{The balance proof}

As always in geometric set theory, the most important part of the analysis of the poset is the classification of balanced virtual conditions. For this, I need to find a precise criterion as to which partial colorings can be extended to conditions in $P_\Gamma$. Define a partial $\Gamma$-coloring $c\colon X\to\gw$ to be \emph{good} if $\dom(c)=F^d$ for some real closed subfield $F\subset\mathbb{R}$, and for every uncountable algebraic set $A\subset X$ visible from $F$, the set $\{c(y)\colon\forall x\in A\ y\mathrel\Gamma x\}$ belongs to the ideal $\mathcal{I}$.

\begin{proposition}
Let $F\subset\mathbb{R}$ be a countable real closed subfield. For a $\Gamma$-coloring $c\colon F^d\to\gw$, the following are equivalent:

\begin{enumerate}
\item $c$ is good;
\item there is $p\in P$ such that $c\subset p$.
\end{enumerate}
\end{proposition}

\begin{proof}
(2) trivially implies (1). For the opposite implication, fix a good coloring $c$.  Define a function $G\colon X\to\mathcal{I}$ by letting $G(x)=0$ if $x\in\dom(c)$, and $G(x)=\{c(y)\colon\forall x\in A\ y\mathrel\Gamma x$ and $A$ is the smallest algebraic set visible in $F$ containing $x\}$. Use Proposition~\ref{listproposition} with $C=F$ to find a coloring $p'$. Let $p$ be the coloring with the same domain as $p'$ and same values except for all $y\in\dom(c)$ let $p(y)=c(y)$. This is the desired common lower bound of the set $a$. Just like in Corollary~\ref{keycorollary}, this is a condition. It stands witness to item (2).
\end{proof}

\noindent For a total good $\Gamma$-coloring $c\colon X\to\gw$, let $\tau_c$ be the $\coll(\gw, X)$-name for the supremum of all conditions $p\in P_\Gamma$ such that $c\subset p$. By the proposition, this is a supremum of a nonempty set.

\begin{proposition}
\label{balanceproposition}
Let $m, n\geq 1$ be numbers. Let $\Gamma=\bigcup_i\Gamma_i$ be a countable union of algebraic graphs on a Euclidean space $X$ such that no graph $\Gamma_i$ contains a perfect biclique (resp. an injective homomorphic copy of $K_{m, n}$). Then

\begin{enumerate}
\item for every total good $\Gamma$-coloring $c\colon X\to\gw$, the pair $\langle \coll(\gw, X), \check c\rangle$ is balanced (resp. $K_{m, n}$-balanced) in $P_\Gamma$;
\item for every balanced pair $\langle Q, \tau\rangle$ there is a total good $\Gamma$-coloring $c$ such that the balanced pairs $\langle \coll(\gw, X), \check c\rangle$ and $\langle Q, \tau\rangle$ are equivalent;
\item distinct colorings yield inequivalent balanced pairs.
\end{enumerate}
\end{proposition}

\begin{proof}
For (1), suppose that $V[G_0]$ and $V[G_1]$ are mutually generic extensions. I first claim that there are no points $x_0\in X\cap V[G_0]\setminus V$ and $x_1\in X\cap V[G_1]\setminus V$ which are $\Gamma$-connected. Suppose towards contradiction that this fails. Find $i\in\gw$ such that $x_0, x_1$ are $\Gamma_i$-connected. It follows from Proposition~\ref{perfproposition} that $\Gamma_i$ contains a perfect biclique.  This contradicts the initial assumptions on the algebraic graph $\Gamma_i$. 

For the $K_{m, n}$-balance statement in (1), assume that $V[G_0], V[G_1]$ are $K_{m, n}$-transcendental extensions and $p_0\in V[G_0]$, $p_1\in V[G_1]$ are conditions containing $c$ as a subset. For the compatibility, I must show that $p_0\cup p_1$ is a $\Gamma$-coloring. By the initial assumptions on the graphs $\Gamma_i$ and the transcendence assumptions, for any $\Gamma$-connected pair $x_0, x_1\in\dom(p_0\cup p_1)$, it must be the case that one of the points is in $V$. Consequently, either both points belong to $\dom(p_0)$ or to $\dom(p_1)$ ane are assigned different colors as $p_0, p_1$ are both $\Gamma$-colorings.

For (2), first strengthen the pair $\langle Q, \tau\rangle$ so that $Q\Vdash X\cap V\subset\dom(\tau)$. By a balance argument, for each point $x\in X$ there must be a specific number $c(x)\in\gw$ such that $Q\Vdash\tau(\check x)=c(x)$. It is immediate that $c\colon X\to\gw$ is a good total $\Gamma$-coloring and $Q\Vdash\check c\subset\tau$. It follows from \cite[Proposition 5.2.6]{z:geometric} that the balanced pairs $\langle \coll(\gw, X), \check c\rangle$ and $\langle Q, \tau\rangle$ are equivalent.

(3) is immediate. 
\end{proof}

\begin{proposition}
Let $m, n\geq 1$ be numbers. Let $\Gamma=\bigcup_i\Gamma_i$ be a countable union of algebraic graphs on a Euclidean space $X$ such that no graph $\Gamma_i$ contains a perfect biclique (resp. an injective homomorphic copy of $K_{m, n}$). Then the poset $P_\Gamma$ is balanced (resp. $K_{m, n}$-balanced).
\end{proposition}

\begin{proof}
Let $p\in P_\Gamma$ be a condition; by Proposition~\ref{balanceproposition}, it is enough to find a good total coloring $c\colon X\to\gw$ extending $p$. To do this, first use the work of Schmerl to see that the chromatic number of $\Gamma$ is countable. \cite[Proposition 1.3]{schmerl:avoidable} shows that an algebraic graph has countable chromatic number if and only if it does not contain an uncountable clique, and then \cite[Theorem 0.1]{schmerl:avoidable} shows that there is a coloring which works for any countable collection of such graphs simultaneously.  Fix a total $\Gamma$-coloring $b\colon X\to\gw$. Find a set $a\subset\gw$ in the ideal $\mathcal{I}$ which is not modulo finite covered by any of the sets $\gw\setminus p''A$ where $A\subset X$ is an uncountable algebraic set visible from $\supp(p)$. Let $a_n\colon n\in\gw$ be pairwise disjoint subsets of $a$ with the same property. Now, define the coloring $c$ by letting $c(x)=p(x)$ for $x\in\dom(p)$, and letting $c(x)\in a_{b(x)}$ be some element which belongs to $p''A$, whenever $x\in X\setminus\dom(p)$ is a point and $A\subset X$ is the smallest algebraic subset of $X$ visible from $\supp(p)$ which contains $x$. 

First, I claim that $c$ is a $\Gamma$-coloring. If $x, y\in X$ are $\Gamma$-connected points and both belong to $\dom(p)$, they receive different colors as $p$ is a $\Gamma$-coloring. If they both belong to $X\setminus\dom(p)$, then they receive different colors as $b$ is a coloring. Finally, if $y\in \dom(p)$ and $x\notin\dom(p)$ and $i\in\gw$ is such that $y\mathrel\Gamma_i x$ holds, then $x$ belongs to the algebraic $\Gamma_i$-neighborhood $A$ of $y$, and by the definition of $c(x)$, it receives a color which $p$ used on $A$ and therefore a color different from $p(y)=c(y)$.

To see why $c$ is good, suppose that $B\subset X$ is an algebraic set which is uncountable and minimal such with respect to inclusion. Consider the set $C=\{y\in X\colon\forall x\in B\ y\mathrel\Gamma x\}$; I must prove that $c''C\in\mathcal{I}$. To this end, let $C_0=C\cap\dom(p)$ and $C_1=C\setminus\dom(p)$. It is clear that $c''C_1\subset a\in\mathcal{I}$. For a point $y\in C_0$, use the minimality of $B$ to argue that there must be $i\in\gw$ such that $B$ is a subset of the $\Gamma_i$-neighborhood of $y$--otherwise each of the neighborhoods has countable intersection with $B$, contradicting the assumption that $y\in C_0$. This means that writing $A\subset X$ for the smallest algebraic set visible from $\supp(p)$ such that $A\subset B$, $y$ is $\Gamma_i$-connected to all elements of $A$, and as such $c(y)$ must belong to the $\mathcal{I}$-small set $\gw\setminus p''A$ which does not depend on the choice of $y\in C_0$. Thus, $c''C\in\mathcal{I}$ holds as desired.
\end{proof}

\bibliographystyle{plain}
\bibliography{odkazy,shelah,zapletal}

\end{document}